\documentclass[11pt, oneside]{amsart}

\usepackage{wrapfig,mathabx}
\usepackage{tikz}
\usepackage{comment}
\usepackage{mathrsfs}
\usepackage{tabularx, hyperref}
\usepackage{amssymb} \usepackage{amsfonts}
\usepackage{amsmath}
\usepackage{amsthm} \usepackage{epsfig, subfig}
\usepackage{ amscd, amsxtra, latexsym}
\usepackage{epsfig,  graphicx, psfrag}
\usepackage[all]{xy}
\usepackage{caption}
\usepackage{enumerate}
\usepackage{color}
\usepackage{amssymb, amsmath,amsthm, mathtools}

\DeclareMathOperator{\Hom}{Hom}

\DeclareMathOperator{\id}{Id}

\newtheorem{lemma}{Lemma}[section]

\newtheorem{prop}[lemma]{Proposition}

\newtheorem*{prop*}{Proposition}
\newtheorem{prop_intro}{Proposition}
\newtheorem{thm_intro}[prop_intro]{Theorem}

\theoremstyle{definition}

\newtheorem{rem}[lemma]{Remark}

\theoremstyle{definition}

\definecolor{darkgreen}{cmyk}{1,0,1,.2}

\newcommand{\R} {\ensuremath {\mathbb{R}}}

\newcommand{\calW} {\ensuremath {\mathcal{W}}}

\newcommand{\calU} {\ensuremath {\mathcal{U}}}
\newcommand{\calV} {\ensuremath {\mathcal{V}}}

\renewcommand{\phi}{\varphi}

\newcommand{\agg}[1]{{\color{red} \textsf{[[AM: #1]]}}}

\begin{document}

\title[]{A remark on the Mayer-Vietoris \\double complex for singular cohomology}

\author[]{R. Frigerio}
\address{Dipartimento di Matematica, Universit\`a di Pisa, Largo B. Pontecorvo 5, 56127 Pisa, Italy}
\email{roberto.frigerio@unipi.it}
\thanks{}

\author[]{A. Maffei}
\address{Dipartimento di Matematica, Universit\`a di Pisa, Largo B. Pontecorvo 5, 56127 Pisa, Italy}
\email{andrea.maffei@unipi.it}
\thanks{}

\begin{abstract}
Given an open cover of a paracompact topological space $X$, there are two natural ways to construct a map from the
cohomology of the nerve of the cover to the cohomology of $X$. 
One of them is based on a partition of unity, and is more topological in nature, while the other one relies
on  the Mayer-Vietoris double complex, and has a more  algebraic flavour.
In this paper we
prove that these two maps coincide, thus answering a question posed by N.~V.~Ivanov.
\end{abstract}
\maketitle

Let $X$ be a paracompact space, and let $\calU=\{U_i\}_{i\in I}$ be an open cover of $X$. We denote by $N(\calU)$ the \emph{nerve} of $\calU$, i.e.~the simplicial set having $I$ as set of vertices, in which a finite subset
$\{i_0,\ldots,i_k\}\subseteq I$ spans a simplex if and only if $U_{i_0}\cap\ldots\cap U_{i_k}\neq \emptyset$. As usual, we endow the geometric realization $|N(\calU)|$ 
of $N(\calU)$ with the weak topology associated to the natural CW structure of $|N(\calU)|$.

Any partition of unity $\Phi=\{\varphi_i\colon X\to\R\}_{i\in I}$ subordinate to $\calU$ induces a map
$$
f_\Phi\colon X\to |N(\calU)|\, ,\qquad f_\Phi(x)=\sum_{i\in I} \varphi_i(x)\cdot i\ .
$$
Moreover, the homotopy class of $f_\Phi$ does not depend on the chosen partition of unity $\Phi$. 
Indeed, if $\Psi$ is another partition of unity, then we have a well-defined homotopy $tf_\Psi+(1-t)f_\Phi$ between $f$ and $g$. 
Therefore, if $R$ is any ring with unity, the map $f_\Phi$ induces
a map
$$
f^*=f_\Phi^*\colon H^*(|N(\calU)|,R)\to H^*(X,R)\ ,
$$
which does not depend on the choice of $\Phi$. 
Throughout this paper, we fix a ring with unity $R$, and for any topological space $Y$ we denote by $C^*(Y)=C^*(Y,R)$ (resp.~$H^*(Y)=H^*(Y,R)$)
the singular cochain complex (resp.~the singular cohomology algebra) of $Y$ with coefficients in $R$. 

There is another natural way to define a map from the (simplicial) cohomology of $N(\calU)$ to the singular cohomology of $X$.  
Let $C^{*,*}(\calU)$ be the Mayer-Vietoris
double complex associated to $\calU$, i.e.~for every $(p,q)\in\mathbb{N}^2$ let 
$$
C^{p,q}(\calU)=\prod_{\underline{i}\in I_p} C^q(U_{\underline{i}})\ ,
$$
where $I_p$ denotes the set of ordered $(p+1)$-tuples $(i_0,\ldots,i_p)\in I^{p+1}$ such that
$
U_{\underline{i}}:=U_{i_0}\cap\ldots\cap U_{i_p}\neq \emptyset
$ (in particular, $I_0=\{i\in I\, |\, U_i\neq \emptyset\}$). 
We refer the reader to Section~\ref{double:sec} for the precise
definition of this double complex. 

To the double complex $C^{*,*}(\calU)$ there is associated 
the \emph{total complex} $T^*$, and we have maps
$$\alpha_X\colon H^*(X)\to H^*(T^*)\, \ ,\qquad\beta\colon H^*(N(\calU))\to H^*(T^*)$$
from the singular cohomology of $X$ to the cohomology of $T^*$ and from the simplicial cohomology of $N(\calU)$ to the cohomology of $T^*$.
Moreover, the map $\alpha$ turns out to be an isomorphism (see Section~\ref{double:sec}).

Let now
$\nu\colon H^*(|N(\calU)|)\to H^*(N(\calU))$ be the canonical isomorphism between the simplicial cohomology of $N(\calU)$ and the singular cohomology of its geometric realization
(see Section~\ref{simpl:sing}).
By setting $\eta=\alpha_X^{-1}\circ \beta\circ \nu$ we have thus defined a map
$$
\eta\colon H^*(|N(\calU)|)\to H^*(X)\ .
$$

The main result of this paper shows that the maps $f^*$ and $\eta$ coincide:

\begin{thm_intro}\label{main:thm}
The maps
$$
f^* \colon H^*(|N(\calU)|)\to  H^*(X)\, ,\quad \eta\colon H^*(|N(\calU)|)\to  H^*(X)
$$
coincide.
\end{thm_intro}

Theorem~\ref{main:thm} answers a question posed by Ivanov in~\cite[page 1113]{Ivanov} and in~\cite[page 71]{ivanov3}.

\section{The Mayer-Vietoris double complex}\label{double:sec}
Let $\calU=\{U_i\}_{i\in I}$ be an open cover of the topological space $X$. 
We now thoroughly describe the horizontal and the vertical differentials of the double complex $C^{*,*}(\calU)$ defined in the introduction, also fixing the notation we will need later.

If $\varphi\in C^{p,q}(\calU)$ and $\underline{i}\in I_p$, then we   denote by $\varphi_{\underline{i}}$ the projection of $\varphi$
on $C^q(U_{\underline{i}})$.
For every $(p,q)\in\mathbb{N}^2$ we denote by 
$$
\delta^{p,q}_v\colon C^{p,q}(\calU)\to C^{p,q+1}(\calU)
$$ 
the ``vertical'' differential which restricts to the usual differential $C^{q}(U_{\underline{i}})\to C^{q+1}(U_{\underline{i}})$ for every $\underline{i}\in I_p$,
and by 
$$
\delta_h^{p,q}\colon C^{p,q}(\calU)\to C^{p+1,q}(\calU)
$$
the ``horizontal'' differential such that, for every $\underline{i}=(i_0,\ldots,i_{p+1})\in I_{p+1}$ and every $\varphi\in C^{p,q}(\calU)$, 
\begin{equation}\label{eq:dh}
(\delta_h^{p,q}(\varphi))_{\underline{i}}=\sum_{k=0}^{p+1}(-1)^k \big(\varphi_{(i_0,\ldots,\hat{i}_k,\ldots,i_{p+1})}\big)|_{U_{\underline{i}}}\ .
\end{equation}
We augment the double complex $C^{*,*}(\calU)$ as follows. We define $C_{q}^\calU$ as the subcomplex of the singular  
chain complex $C_q(X)$ generated 
(over $R$) by those singular
simplices $s\colon \Delta^q\to X$ such that $s(\Delta^q)$ is contained in $U_i$ for some $i\in I$. We then set 
$C^{-1,q}(\calU)=C^{q}_\calU=\Hom (C_{q}^\calU,R)$. The usual boundary maps of the complex $C_*^\calU$ induce dual coboundary maps, which
endow $C^{*}_\calU$ with the structure of a complex. The inclusion of the complex $C_*^\calU$ in the full complex of singular chains induces a map
of complexes $\widetilde \gamma: C^*(X)\to C^*_\calU$. It is known that the map $\gamma$ induced in cohomology 
is an isomorphism (see e.g.~\cite[Proposition 2.21]{Hat}) and we will identify the singular cohomology of $X$ with the cohomology
of the complex $C^{q}_\calU$ via $\gamma$. 
The augmentation maps 
$\delta^{-1,q}\colon C^{-1,q}(\calU)\to C^{0,q}(\calU)$ are defined by setting, for every $i\in I_0$,
$$
(\delta^{-1,q}(\varphi))_i=\varphi|_{U_i}\ .
$$

In order to define the augmentation of the vertical complexes, we consider the Cech complex given by $C^{p,-1}(\calU)=\widecheck{C}^p(\calU)=\prod_{\underline{i}\in I_p} R$, 
with boundary maps defined as in formula \eqref{eq:dh}. We then define the augmentation maps 
$\delta^{p,-1}:{C}^{p,-1}(\calU)\to C^{p,0}(\calU)$ by setting 
$$
(\delta^{p,-1} (\varphi))_{\underline{i}}(s)=\varphi_{\underline{i}}\ \in\ R
$$
for every $\varphi\in {C}^{p,-1}(\calU)$, every $\underline{i}=(i_0,\ldots,i_p)\in I_p$ and every singular simplex 
$s\colon \Delta^0\to U_{i_0}\cap\ldots \cap U_{i_p}$.

\begin{rem}
The complex $\widecheck{C}^*(\calU)$ computes the {Cech cohomology} of the cover $\calU$ with coefficients in the \emph{constant} presheaf $R$. 
Such cohomology, which is usually denoted by $\widecheck{H}(\calU)$, 
is tautologically isomorphic to the simplicial cohomology of the nerve $N(\calU)$. 
It is costumary
to rather study the Cech cohomology of $\calU$ with coefficients in the \emph{locally constant} sheaf $R$. However
this cohomology does not always coincide with the cohomology
of $N(\calU)$. They coincide, for example, under the assumption that every $U_{\underline{i}}$, $i\in I_p$, $p\in\mathbb{N}$, is path connected.
\end{rem}

In the next lemma we prove that the rows of the augmented double complex are exact. 
\begin{lemma}\label{exactrows:lemma}
 For every $q\in\mathbb{N}$, the complex
 $$
 \xymatrix{
 0 \ar[r] & C^{-1,q}(\calU) \ar[r]^{\delta_h^{-1,q}} & C^{0,q}(\calU) \ar[r]^{\delta_h^{0,q}}  & \cdots \ar[r]^{\delta_h^{p-1,q}} & C^{p,q}(\calU) \ar[r]^{\delta_h^{p,q}} & \cdots
 }
 $$
is exact.
 \end{lemma}
\begin{proof}
 Let $s\colon \Delta^q\to X$ be a singular simplex such that $s(\Delta^q)$ is contained in $U_i$ for some $i\in\ I$. 
We set $$C^{-1,q}_s(\calU)=\{\varphi\in C^{-1,q}(\calU)\, |\, \varphi(s')=0\ \text{for\ every}\ s'\neq s\}\ ,$$ and for every $p\geq 0$ and every $\underline{i}\in I_p$  we define
$$
C^{p,q}_s(U_{\underline{i}})=\{\varphi\in C^{q}(U_{\underline{i}})\, |\, \varphi(s')=0\ \text{for\ every}\ s'\neq s\}\ .
$$
We also set 
 $I(s)=\{i\in I\, |\, s(\Delta^q)\subseteq U_i\}$, $I_p(s)=(I(s))^{p+1}\subseteq I_p$, and
$$
C^{p,q}_s(\calU)=\prod_{\underline{i}\in I_p(s)} C^{p,q}_s(U_{\underline{i}})
$$
(according to our definition, $C_s^{p,q}(U_{\overline{i}})=0$ whenever $\underline{i}\notin I_p(s)$). 
We observe that $C^{*,q}_s(\calU)$ is a subcomplex of $C^{*,q}(\calU)$, and that
$$C^{p,q}(\calU)=\prod_{s\colon \Delta^q\to X}  C^{p,q}_s(\calU)\ .$$
Hence, in order to conclude it is sufficient to show that each $C^{*,q}_s(\calU)$ is exact. However,
the complex $C^{*,q}_s(\calU)$ is isomorphic to the simplicial cohomology complex of the full simplex
with vertices $I(s)$, whence the conclusion.
\end{proof}

As a consequence of the previous lemma the cohomology groups of the complex $C^{-1,*}$ are isomorphic to 
the cohomology of the \emph{total complex} $T^*$ associated to the double complex. Recall that  $T^*$ is defined by setting
$$
T^n=	\bigoplus_{\substack{(p,q)\in\mathbb{N}^2 \\ p+q=n}} C^{p,q}(\calU)
$$
with differential
$ \delta^n\colon T^n\to T^{n+1}$ given by $\delta^n=\bigoplus_{p+q=n} (\delta_h^{p,q}+(-1)^p \delta_v^{p,q})
$. 
 The augmentation maps induce morphisms of complexes $\widetilde \alpha^*: C^{*}_\calU\to  T^*$ and
$\widetilde \beta^*\colon\widecheck C^*\to T^*$ and we denote by $\alpha$, $\beta$ the maps induced by $\alpha^*$, $\beta^*$ on cohomology.
By Lemma \ref{exactrows:lemma} 
$\alpha$ is an isomorphism in every degree  
and the  map $\alpha\circ\gamma\colon H^*(X)\to H^*(T^*)$ is the isomorphism $\alpha_X$ defined in the introduction.
We define $\zeta=\alpha^{-1}\circ\beta$ and $\eta=\alpha_X^{-1}\circ\beta\circ\nu$.

The notation introduced so far is summarized in the following diagram:
$$
\xymatrix{H^*(X) \ar^{\simeq}_{\gamma}[rr] \ar@/^1.5pc/[rrr]^{\alpha_X}& & H^*(C^{q}_\calU)\ar^{\simeq}_{\alpha}[r] & H^*(T) \\
& & & \widecheck{H}^*(\calU)=H^*(N(\calU))\ar_{\beta}[u]\ar[ul]^{\zeta}\\
& & & H^*(|N(\calU)|)\ar[uulll]^{\eta}\ar[u]_{\nu}\ .}
$$
When we want to stress the dependence of these constructions on the cover $\calU$ we write $\alpha_\calU$, $\beta_\calU$, etc.


\section{The case of a simplicial complex}\label{simpl:sing}

In this section we analyze the Mayer-Vietoris double complex when $X=|S|$ is the geometric realization of a simplicial complex $S$.
Let $I$ be the vertex set of $S$. We
consider the open cover $\calU^*=\{U^*_i\}_{i\in I}$ of $|S|$ given by the open stars of the vertices, i.e.~ for every $i\in I$ we set
$U_i=\{x\in |S|: x_i>0\}$, where $x_i$ denotes the barycentric coordinate of the point $x$ relative to the vertex $i$.
Observe  that the simplical complexes $N(\calU^*)$ and $S$ on the set of vertices $I$ are equal and we will identify them. 
Hence, in this case $\eta_{\calU^*}:H^*(|S|)\to H^*(|S|)$. Notice also that in this case all intersections $U_{\underline i}^*$ are contractible,
hence, also the columns of the augmented double complex are exact. As a consequence, $\beta$ and $\zeta$ are isomorphisms. The next proposition
shows that the map $\eta$ is the identity in this case.

\begin{prop}\label{prp:simplicialcase} If $S$ is a simplicial complex and $\calU^*$ is the  
cover described above then $\eta=\text{Id}$.
\end{prop}

To prove this proposition we will perform a computation by describing a lift of $\zeta$ at the level of cochains.
To simplify the computations we will use \emph{alternating} cochains, whose definition is recalled below. 


\subsubsection*{Construction of $\widetilde \zeta$}
We start by describing a lift 
$$ \widetilde{\zeta}\colon \widecheck{C}(\calU)\to C^{-1,p}(\calU)=C^p_\calU$$
of the map $\zeta$ at the level of cochains.
We first construct chain homotopies
$$ K^{p,q}\colon C^{p,q}(\calU)\to C^{p-1,q}(\calU)\, , \quad p\geq 0\, ,\quad q\geq 0\ .$$
For each singular simplex $s$ 
with image contained in some open subset $U_i$ we fix an index $i(s)$ such that $\text{Im}\, s\subseteq U_{i(s)}$. 
For all $\varphi\in C^{p,q}(\calU)$ and for all singular simplices $s$ with image contained in 
$U_{\underline{i}}$ for some $\underline{i}\in I_{p-1}$, $p\geq 0$, we define 
$$
\big(K^{p,q}(\varphi)_{\underline{i}} \big)(s)=
\varphi_{i(s),\underline{i}}(s) 
$$
(when $p=0$ there is no index $\underline{i}$ and we just take $s\in C_q^{\calU}$). It is easy to check that 
$
\delta_h^{p-1,q}K^{p,q}+K^{p+1,q}\delta_h^{p,q}=\id
$
for every $p\geq 0$, $q\geq 0$. Hence, if we define 
$$
\widetilde{\zeta}=(-1)^{\frac{p(p+1)}{2}} \, K^{0,p}\circ\delta_v^{0,p-1}\circ K^{1,p-1} \circ \cdots \circ K^{p-1,1}\circ \delta_v^{p-1,0} \circ K^{p,0}\circ \delta_v^{p,-1}
$$
then for every cocycle $\varphi\in \widecheck{C}^{p}(\calU)$ we have $\zeta([\varphi])=[\widetilde \zeta(\varphi)]$ in $H^p(C^*_\calU)$.


\subsubsection*{Singular and algebraic simplices}
Let us now recall the construction of the isomorphism $\nu$ between the simplicial cohomology $H^*(S)$ of $S$ and the singular cohomology
$H^*(|S|)$ of its geometric realization. Let $C_*(S)$ be the chain complex of simplicial chains on $S$, 
i.e.~let $C_p$ be the free $R$-module with basis
$$\{(i_0,\ldots, i_p)\in I^{p+1}\, |\, \{i_0,\ldots,i_p\}\ \text{is\ a\ simplex\ of}\ S\}\ ,$$
and let $C^*(S)$ be the dual chain complex of $C_*(S)$. Elements of the basis just described are usually called \emph{algebraic} simplices.

For any algebraic simplex $\sigma=(i_0,\ldots,i_p)$ of $S$, one defines the singular simplex
$
\langle \sigma\rangle\colon \Delta^p\to |S| 
$
by setting
$$
\langle \sigma \rangle (t_0,\ldots,t_p)=t_0i_0+\dots+t_pi_p\ .
$$
The map $\sigma\mapsto \langle \sigma\rangle$ extends to a chain map $C_*(S)\to C_*(|S|)$, whose dual map $\widetilde \nu:C^*(|S|)\to C^*(S)$ induces the 
isomorphism $ \nu\colon H^*(|S|)\to H^*(S)$
(see e.g.~\cite{Hat} Theorem 2.27). We write $\nu_S,\widetilde \nu_S$ when we want to stress the dependence on the simplicial complex.


\subsubsection*{Alternating cochains}To compute $\zeta\circ \nu$  it is convenient to use \emph{alternating} cochains. 
Let $\mathfrak{S}_{p+1}$ be the permutation group of $\{0,\dots,p\}$. We say that a simplicial cochain $\varphi\in C^p(S)$ is alternating if 
$\varphi(i_{\tau(0)},\dots,i_{\tau(p)})=\varepsilon(\tau)\varphi(i_0,\dots,i_p)$
for every $\tau\in\mathfrak{S}_{p+1}$, and $\varphi(i_0,\dots,i_p)=0$ whenever $i_j=i_{j'}$ for some $j\neq j'$. 
Alternating cochains form a subcomplex of 
the complex of cochains which is homotopy equivalent to the full complex (see e.g.~\cite[Chap.20, Section 23]{StPr}).

Alternating cochains may be defined also in the context of singular homology as follows. 
For every  $\tau\in\mathfrak{S}_{p+1}$  denote by $\rho_\tau\colon \Delta^p\to \Delta^p$ the 
affine automorphism of $\Delta^p$ defined by $\rho_\tau(t_0,\dots,t_p)=(t_{\tau(0)},\dots,t_{\tau(p)})$. 
If $X$ is a topological space,
we say that a singular cochain
$\varphi\in C^p(X)$ is \emph{alternating} if $\varphi(s\circ \rho_\tau)=\varepsilon(\tau)\, \varphi(s)$ for every $\tau\in\mathfrak{S}_{p+1}$ 
and every singular simplex $s\colon \Delta^p\to X$,
and $\varphi(s)=0$ for every singular simplex $s$ such that $s=s\circ \rho_\tau$ for an odd permutation $\tau\in\mathfrak{S}_{p+1}$. 
Alternating singular cochains form a subcomplex the complex of singular cochains which 
is homotopy equivalent to the full complex (see e.g.~\cite{bar}). 
Moreover, the map $\widetilde{\nu}$ introduced above sends alternating singular cochain to alternating simplicial cochains,
and both the homotopy maps $K^{p,q}$ and the vertical differential send alternating cochains to alternating ones.

\medskip

We want to compute $\widetilde{\zeta}(\varphi)$ on singular simplices of the form 
$\langle\sigma\rangle$, as $\sigma$ varies among the algebraic simplices of $S$. 
However, simplices of $S$ are not contained in any $U_i^*$. 
We will then make use of the barycentric subdivision $S'$ of $S$, together with a suitable simplicial approximation of the identity $S'\to S$. 
Let $I'$ be the set of vertices of $S'$. This set is in bijective correspondence with the set of
simplices of $S$: for $i'\in I'$ we denote by $\Delta_{i'}$ the simplex of 
$S$ of which $i'$ is the barycenter; in the opposite direction, if $\Delta$ is a simplex of $S$ we denote by $i'_\Delta$ its barycenter. 
The $p$-simplices of $S'$ are then the subsets $\{i_0',\dots,i_p'\}$ where
$\Delta_{i_0'}\subset \dots \subset \Delta_{i_p'}$.

If for every simplex $\Delta$ of $S$ we denote by $b_\Delta\in |S|$ the geometric barycenter of  $\Delta$ 
then the map $b\colon |S'|\to |S|$ defined by $b(\sum_{\Delta }t_{\Delta}i'_{\Delta})=\sum_{\Delta} t_\Delta b_\Delta$ is a homeomorphism,
and we will identify the geometric realization of $S'$ and $S$ via this map. We construct a second map from $|S'|$ to $|S|$ as follows.
We fix an auxiliary total ordering on $I$, and  we define a simplicial map $g\colon S'\to S$
by setting
$$g(i')=\max \Delta_{i'} $$ 
for every vertex $i'$ of $S'$. The geometric realization $|g|\colon |S|=|S'|\to |S|$
of $g$ is homotopic to $b$ via the homotopy $tb+(1-t)|g|$, $t\in [0,1]$. 



We may define the map
$i$ used to construct the homotopies $K^{p,q}$ in such a way that, for every algebraic simplex $\sigma'=(i'_0,\ldots,i'_p)$ of $C_*(S')$,
$$
i(\langle \sigma' \rangle)=\min \{g(i'_0),\dots,g(i'_p)\}\ .
$$
For simplicity, we will denote $i(\langle \sigma' \rangle)$ by $i(\sigma')$.
With this choice, the singular simplex 
$\langle \sigma' \rangle$ is supported in
$U_{i(\sigma')}^*$ as required in the definition of the 
map $i$.

Let $\alpha=(\alpha_{\underline i})\in C^{h,k}(\calU^*)$ and let  $\sigma'=(i'_0,\dots,i'_{k+1})\in  C_{k+1}(U_{\underline i}^*)$, $\underline i\in I_h$, be an  algebraic $(k+1)$-simplex of $S'$. 
If $\partial_h \sigma'=(i'_0,\dots,\widehat{i}'_h,\dots,i'_{k+1})$ denotes the algebraic $h$-th face of $\sigma'$, then
\begin{align}\label{formula}
\begin{split}
\big(  \delta^{h-1,k}_v K^{h,k} (\alpha)  \big) (\langle \sigma' \rangle)
 & =\sum_{h=0}^{k+1} (-1)^h \big(K^{h,k}(\alpha)\big)_{\underline i} (\langle\partial_h \sigma'\rangle) \\
 & =\sum_{h=0}^{k+1} (-1)^h \alpha_{i(\partial_h\sigma'),\underline i} (\langle\partial_h \sigma'\rangle )\ .
 \end{split}
\end{align}



\begin{lemma}\label{key:lemma}
Let $\varphi$ be an alternating cocycle in $C^p(N(\calU^*))=\widecheck{C}^{p}(\calU^*)$, and let $\sigma'\in C_p(S')$ be an algebraic simplex.
Then
$$
\big(\widetilde\zeta(\varphi)\big)(\langle\sigma'\rangle)=\varphi(g_*(\sigma'))\ ,
$$
where $g_*\colon C_p(S')\to C_p(S)$ is the map induced by $g\colon S'\to S$. 
\end{lemma}

\begin{proof} Let $\sigma'=(i'_0,\dots,i_p')$ and set $\Delta_\ell=\Delta_{i'_\ell}$ for $\ell=0,\dots,p$ and
$i_\ell=g(i'_\ell)$. 
Recall that simplices of $S'$ corresponds to comparable subsets of a simplex of $S$. Moreover,
since $\varphi$ is alternating, both $g^*(\varphi)$ and $\widetilde{\zeta}(\varphi)$ are alternating, thus in order to check that the equality of the statement holds we may assume that 
$$
\Delta_0\subsetneq \Delta_1 \dots \subsetneq \Delta_p\ .
$$
By definition we have $i_\ell=\max \Delta_\ell$, hence in particular
$i_0\leq i_1\leq\dots\leq i_p$. 
Since $\varphi$ is alternating, this implies at once that
\begin{equation}\label{phig}
\varphi(g_*(\sigma'))=\begin{cases}
\varphi_{i_0,i_1,\dots,i_p} &\text{if } i_0<\dots<i_p \\
 0 &\text{otherwise} .
 \end{cases}
\end{equation}

Let us now compute $\big(\widetilde\zeta(\varphi)\big)(\langle\sigma'\rangle)$. 
For every algebraic simplex $\tau_k'\in C_k(S')$, we write $\tau_{k-1}'<\tau'_k$ if $\tau_{k-1}'$ is an algebraic face of $\tau'_k$, i.e.~if
there exists $h=0,\ldots,k$ such that $\tau'_{k-1}=\partial_h \tau'_k$. 
By iterating~\eqref{formula} we  get
\begin{equation}\label{zeta}
\big(\widetilde\zeta(\varphi)\big)(\langle\sigma'\rangle)=
(-1)^{\frac{p(p+1)}2}\sum_{\sigma'_0< \dots < \sigma'_p=\sigma'} \pm \varphi_{i(\sigma'_0),i(\sigma'_1),\dots ,i(\sigma'_p)}\ .
\end{equation}
Let now $\sigma'_0< \dots < \sigma'_p$ be a fixed descending sequence of faces of $\sigma'$. 
Since the map $i$ is given by taking a minimum we have $i(\sigma'_0)\geq i(\sigma'_1)\geq \dots \geq i(\sigma'_p)$ and all these elements belong to the set $\{i_0,\dots,i_p\}$.  Hence if $\varphi_{i(\sigma'_0),i(\sigma'_1),\dots ,i(\sigma'_p)}\neq 0$ we have $i_0<\dots<i_p$ and $i (\sigma'_\ell)=i_{p-\ell}$ for every $\ell$. In particular $\big(\widetilde\zeta(\varphi)\big)(\langle\sigma'\rangle)$ agrees with $\varphi(g_*(\sigma'))$ in the second case of formula \eqref{phig}.

Assume now $i_0<\dots<i_p$. As just observed, if $\varphi_{i(\sigma'_0),i(\sigma'_1),\dots ,i(\sigma'_p)}\neq 0$ then
 $i (\sigma'_\ell)=i_{p-\ell}$ for every $\ell$, and this readily implies that 
the unique 
non-trivial addend in the right-hand sum in~\eqref{zeta} corresponds to the sequence 
$$
\overline{\sigma}'_{0}=(i'_p),\quad  \overline{\sigma}'_{1}=(i'_{p-1},i'_{p}),  \quad \dots\quad ,\overline{\sigma}'_{p}=(i'_0,\dots, i'_{p-1},i'_{p})\ .
$$
In particular, for every $j=0,\dots,p-1$ we have $\overline{\sigma}'_{j}=(-1)^0\partial_0 \overline{\sigma}'_{j+1}$. Hence
\begin{gather*}
\big(\widetilde\zeta(\varphi)\big)(\langle\sigma'\rangle)=
(-1)^{\frac{p(p+1)}2}\varphi_{i(\overline{\sigma}'_0),i(\overline{\sigma}'_1),\dots ,i(\overline{\sigma}'_p)}\\
=(-1)^{\frac{p(p+1)}2}\varphi_{i_p,i_{p-1},\dots,i_0}
=\varphi_{i_0,i_{1},\dots,i_p}
\end{gather*}
settling also the first case in formula \eqref{phig}.
\end{proof}

Before proving the proposition we notice that the map $C_*(S)\to C_*(|S|)$ constructed above does not factor through 
$C^{\calU^*}_*$ because no positive-dimen\-sion\-al simplex of $S$ is contained in $U_i^*$ for any $i\in I$. However
the analogous map from $C_*(S')$ to $C_*(|S|)$ does. Hence the map $\widetilde \nu_{S'}\colon C^{*}(|S'|)\to C^*(S')$ factors
as $\widetilde \nu_{S'}= \widetilde \mu \circ \widetilde \gamma$, where
$\widetilde{\gamma}\colon C^{*}(|S'|)\to C^*_{\calU^*}$ is the map defined in Section~\ref{double:sec}, and
  $\widetilde \mu\colon C_{\calU^*}^* \to C^*(S')$. We denote by $\mu\colon H^*(C_{\calU^*}^*)\to H^*(S)$
  the map induced by $\widetilde{\mu}$ on cohomology.

\begin{proof}[Proof of Proposition \ref{prp:simplicialcase}]
 Being $\nu_{S'}:H^*(|S|)=H^*(|S'|)\to H^*(S')$ injective and $|g|$ homotopic to the identity, in order
 to prove the proposition it is sufficient to show that $\nu_{S'}\circ\eta=\nu_{S'}\circ |g|^*$. Now recall that 
 $\eta=\gamma^{-1}\circ\zeta\circ\nu_S$, hence $\nu_{S'}\circ\eta=\mu\circ\zeta\circ\nu_{S}$. Hence it is enough to prove that
 $\mu(\zeta(\nu_S(c)))=\nu_{S'}(|g|^*(c))$ for all $c\in H^p(|S|)$ or, equivalently, that
 $$\widetilde \mu (\widetilde \zeta(\widetilde \nu_S(\psi)))(\sigma')=\widetilde \nu_{S'}(|g|^*(\psi))(\sigma')$$ where 
 $\psi\in C^p(|S|)$ is a cocycle and $\sigma'$ is any algebraic simplex of $S'$. 
 Morover, as observed above we can choose $\psi$ to be alternating. However, if we set $\phi=\widetilde \nu_{S}(\psi),$ then
\begin{gather*}
\mu(\widetilde \zeta(\widetilde \nu_S(\psi)))(\sigma')= \mu(\widetilde \zeta(\phi))(\sigma')=
\big(\widetilde\zeta(\varphi)\big)(\langle\sigma'\rangle)\ ,\\
\widetilde \nu_{S'}(|g|^*(\psi))(\sigma')=(|g|^*(\psi))(\langle \sigma'\rangle )=\psi(|g|_*(\langle \sigma'\rangle ))=\varphi (g_*(\sigma'))\ ,
\end{gather*}
hence the conclusion follows from Lemma~\ref{key:lemma}.
\end{proof}



\section{Proof of Theorem \ref{main:thm}}

We can now prove the Theorem stated in the introduction. We first notice  that the construction of $\eta$ is compatible with continuous maps in the 
following sense.

\begin{lemma}\label{commutes1}
 Let $h\colon Y\to Z$ be a continuous map, and let $\calV=\{V_i\}_{i\in I}$, $\calW=\{W_i\}_{i\in I}$ be open covers of $Y,Z$, respectively, 
 such that $h(V_i)\subseteq W_i$ for every $i\in I$. The identity of the set $I$ extends to a simplicial map $N(h)\colon N(\calV)\to N(\calW)$, 
 and in particular it induces a continuous map $\hat h:|N(\calV)|\to |N(\calW)|$. Then the following diagram commutes:
 $$
 \xymatrix{
 {H}^*(|N(\calW)|)\ar[r]^{\hat h^*}\ar[d]^{\eta_\calW} & {H}^*(|N(\calV)|)\ar[d]^{\eta_\calV}\\
  H^*(Z)\ar[r]^{h^*}& H^*(Y)\ .
 }
 $$
\end{lemma}
\begin{proof}
 By considering the restriction of $h$ to the open subset $V_i$ the map $h$ 
 induces a morphism $\{h^{p,q}\}$ between the double complex associated to $\calW$ and the double complex associated to $\calV$ and 
 between their augmentations. Hence we have $\zeta_\calV \circ N(h)^*=h^*\circ \zeta_\calW:H^*(N(\calW))\to H^*(C^*_\calV)$. 
We also have $\widetilde \gamma_\calV \circ h^*= h^{-1,*}\circ \widetilde \gamma_\calW$ and 
 by the definition of the map $\nu$ we have $\nu_\calV \circ h^*= N(h)^*\circ \nu_\calW $. By the definition of $\eta$, these three
 commutations imply the commutativity claimed in the lemma.
\end{proof}

We can now conclude the proof of our main theorem. Let $\calU=\{U_i\}_{i\in I}$ be an open cover of the space $X$, let $N(\calU)$ be the  nerve of $\calU$,
and let  $\calU^*=\{U_i^*\}_{i\in I}$ be the open cover of $|N(\calU)|$ given by the open stars of the vertices of $N(\calU)$. 
Let $f_\Phi\colon X\to |N(\calU)|$ be the map associated to a partition
of unity subordinate to $\calU$ as described in the introduction. We would like to apply the previous lemma to the covers 
$\calU$ of $X$ and  $\calU^*$ of $|N(\calU)|$ and to the map $h=f_\Phi$,
but the containment $f_\Phi(U_i)\subseteq U_i^*$ does not hold in general. Therefore, we consider the cover
$\widetilde{\calU}=\{\widetilde{U}_i\}_{i\in I}$ of $X$ defined by $\widetilde{U}_i=f_\Phi^{-1}(U_i^*)$ for every $i\in I$. 

We can now apply  Lemma~\ref{commutes1} to the map $h=f_\Phi$ and to the covers $\calV=\widetilde\calU$ and $\calW=\calU^*$. 
Since
$\widetilde{U}_i\subseteq U_i$ for every $i\in I$, Lemma~\ref{commutes1} also applies to the case when $h=i_X$ is the identity map of $X$,  and to the covers 
$\calV=\widetilde\calU$ and $\calW=\calU$. Hence we obtain the following commutative diagrams:
$$ \xymatrix{
 H^*(|N(\calU^*)|)\ar[r]^{\hat f_\Phi^*}\ar[d]^{\eta_{\calU^*}} & H^*(|N(\widetilde \calU)|)\ar[d]^{\eta_{\widetilde\calU}} 
 & H^*(|N(\calU)|)\ar^{\eta_\calU}[d]\ar^{\hat i_X^*}[r] & H^*(|N(\widetilde \calU)|)\ar[d]^{\eta_{\widetilde\calU}}\\
 H^*(|N(\calU)|)\ar[r]^{f_\Phi^*}                               & H^*(X)                                          
 & H^*(X) \ar@{=}[r] & H^*(X)\ .}
$$
As already noticed in the previous section the simplicial complexes $N(\calU)$ and $N(\calU^*)$ with set of vertices $I$ are equal 
and, by construction, so are the simplicial  maps $N(i_X)$ and $N(f_\Phi)$ from $N(\widetilde \calU)$ to $N(\calU^*)=N(\calU)$. In particular 
$\hat f_\Phi^*=\hat i_X^*$. Finally
by Proposition \ref{prp:simplicialcase} $\eta_{\calU^*}$ is the identity. Hence 
$$f_\Phi^* = f_\Phi^*\circ\eta_{\calU^*}=\eta_{\widetilde\calU}\circ \hat f_\Phi^*=\eta_{\widetilde\calU}\circ \hat i_X^*=\eta_\calU\ ,
$$
which proves the theorem.

\bibliographystyle{amsalpha}
\bibliography{biblionote}



\end{document}